\pdfoutput=1 
\documentclass[copyright,creativecommons]{eptcs}
\providecommand{\event}{ACT 2021} 
\usepackage{underscore}           

\title{Limits and Colimits in a Category of Lenses}

\author{
Emma Chollet
\institute{ETH Z\"urich\\
Z\"urich, Switzerland}
\email{emma.chollet@eawag.ch}
\and
Bryce Clarke
\institute{Macquarie University\\
Sydney, Australia}
\email{bryce.clarke1@hdr.mq.edu.au}
\and
Michael Johnson
\institute{Macquarie University\\
Sydney, Australia}
\email{mike@ics.mq.edu.au} 
\and
Maurine Songa
\institute{
University of KwaZulu-Natal\\
Durban, South Africa}
\email{maurine@aims.ac.za}
\and 
Vincent Wang
\institute{University of Oxford\\
Oxford, UK}
\email{vincent.wang@cs.ox.ac.uk}
\and
Gioele Zardini
\institute{ETH Z\"urich\\ 
Z\"urich, Switzerland}
\email{gzardini@ethz.ch}
}

\usepackage{mathtools, amssymb}
\usepackage{eucal}

\usepackage{tikz}
\usepackage{tikz-cd}
\usepackage{tikzit}

\usepackage{environ}
\makeatletter
\newsavebox{\measure@tikzpicture}
\NewEnviron{scaletikzpicturetowidth}[1]{%
  \def\tikz@width{#1}%
  \def\tikzscale{1}\begin{lrbox}{\measure@tikzpicture}%
  \BODY
  \end{lrbox}%
  \pgfmathparse{#1/\wd\measure@tikzpicture}%
  \edef\tikzscale{\pgfmathresult}%
  \BODY
}
\makeatother

\usepackage[shortlabels]{enumitem}

\usepackage{subcaption}

\usepackage{amsthm}

\newtheorem{theorem}{Theorem}[section]
\newtheorem{proposition}[theorem]{Proposition}
\newtheorem{lemma}[theorem]{Lemma}
\newtheorem{corollary}[theorem]{Corollary}
\newtheorem*{lemmaA}{Lemma A} 
\newtheorem*{lemmaB}{Lemma B} 

\theoremstyle{definition} 
\newtheorem{definition}[theorem]{Definition}
\newtheorem{example}[theorem]{Example}

\theoremstyle{remark}
\newtheorem*{remark}{Remark}

\newcommand{\Cat}{\mathrm{\mathcal{C}at}}
\newcommand{\Lens}{\mathrm{\mathcal{L}ens}}
\newcommand{\Dopf}{\mathrm{\mathcal{D}opf}}
\newcommand{\Bool}{\mathrm{\mathcal{B}ool}}
\newcommand{\Span}{\mathsf{Span}} 

\newcommand{\op}{\mathrm{op}}

\newcommand{\phibar}{\overline{\varphi}}
\newcommand{\gammabar}{\overline{\gamma}}
\newcommand{\br}{\rightleftharpoons}

\newcommand{\U}{\mathcal{U}} 
\newcommand{\D}{\mathcal{D}}

\DeclareMathOperator{\Ob}{Ob}

\newcommand{\cheap}{\mathrm{cheap}}
\newcommand{\expensive}{\mathrm{expensive}}
\newcommand{\slow}{\mathrm{slow}}
\newcommand{\average}{\mathrm{average}}
\newcommand{\fast}{\mathrm{fast}}
\newcommand{\false}{\texttt{false}}
\newcommand{\true}{\texttt{true}}

\def\titlerunning{Limits and Colimits in a Category of Lenses}
\def\authorrunning{Chollet, Clarke, Johnson, Songa, Wang, Zardini}

\begin{document}
\maketitle

\begin{abstract}
Lenses are an important tool in applied category theory. 
While individual lenses have been widely used in applications, many of 
the mathematical properties of the corresponding categories of lenses
have remained unknown. 
In this paper, we study the category of small categories and asymmetric 
delta lenses, and prove that it has several good exactness properties.
These properties include the existence of certain limits and colimits, 
as well as so-called imported limits, such as imported 
products and imported pullbacks, which have arisen previously in
applications.  
The category is also shown to be extensive, and it has an 
image factorisation system. 
\end{abstract}

\section{Introduction}

Lenses, and their use for synchronising systems, have been an important 
tool in applied category theory dating back to even before the term 
``Applied Category Theory'' was first used in its modern form.
Lenses were introduced by Pierce and Schmitt in 2003 under that name 
\cite{PS03}, but under other names lenses were an important part of the 
database view updating work of the 1980s. 
The full axiomatic description of what are now called 
\emph{very well behaved set-based lenses} first appeared in a study of 
storage management in the thesis of Oles \cite{Ole82}.
Since that time many different flavours of lenses have been introduced, 
and a very wide variety of applications have been found.

The first lenses were \emph{asymmetric lenses}, so called to emphasise that 
when they were used to maintain consistency between two systems, one of 
the systems had all the information required to reproduce the entire 
state of the other system (as in a database and its views).  
However, many real-world synchronisation problems are more symmetric in 
that each system has state that cannot be derived from the other. 
From the beginning of the study of such symmetric systems it was 
recognised that symmetric lenses could be built from asymmetric lenses, 
so the mathematical study of asymmetric lenses has remained central to 
the subject.

The set-based asymmetric lenses were soon seen to be a special case of 
a more general, and more useful, notion called \emph{delta lenses} 
\cite{DXC11}, which might also be described as \emph{category-based lenses}. 
The original set-based lenses are the special case where the categories 
in question are codiscrete \cite{JR16}. 
These asymmetric category-based lenses were seen to unify a wide range of 
lenses and their applications, and they are the subject of study in 
this paper.

Another distinction among lenses worthy of note has sometimes been 
described as the lawful versus the lawless lenses. 
It often happens in engineering that systems are designed with axioms 
or assertions or other rules of well-definedness in mind, but the major 
engineering job is to build the infrastructure which can support those 
systems, and that infrastructure may, or may not, enforce the axioms 
--- it is quite common to leave the questions of validity with respect 
to axioms or assertions to the user.  
Thus we have the \emph{lawless lenses}, those which have the lens 
operations, usually called Put and Get, but with few or no requirements 
about how those operations interact with each other or with data. 
In fact these lawless lenses have come to be seen as important in a 
range of applications of their own including economics, game theory and 
machine learning. 
Nevertheless, the \emph{lawful lenses}, those that are required to satisfy 
the basic axioms originally proposed, axioms which are seen here to 
correspond to various types of functoriality and fibering, remain the 
principal object of mathematical study, and are the lenses analysed 
in this paper.

When we say \emph{lens} in this paper we will mean 
\emph{lawful category-based asymmetric lens}.

The urgency of the applications of lenses has meant over the years 
that much of the work has focused on individual lenses as needed. 
Of course it was recognised early that lenses compose, associatively 
and with identities, and so form a category called $\Lens$, whose 
objects are small categories and whose arrows are lenses. 
But that category has, until this paper, been little studied, and 
its properties were only hinted at in earlier work. 
One of those properties caught the attention of early 
workers, and is an important motivation for this paper.

We have already noted that symmetric lenses can be studied via 
asymmetric lenses: a \emph{symmetric lens} is an equivalence class of spans 
of asymmetric lenses. 
So one might expect that the well-understood theory of spans in a 
category would apply, and would support the study of the (bi)category 
of symmetric lenses as $\Span(\Lens)$.
That theory depends on using pullbacks to compose spans, so the 
obvious first step was to construct pullbacks in $\Lens$. 
Attempts to do this seemed straightforward:  one can calculate the 
pullback of the lenses' Get functors in $\Cat$, and it is easy to find a 
canonical construction of Put operations on the resultant projections which 
satisfy all the required axioms. 
Thus one has a ``pullback'' in $\Lens$, but the quotation marks are 
there because it soon became apparent that most of the ``pullbacks'' 
were not pullbacks in $\Lens$ at all --- they did not satisfy the 
required universal property with respect to lenses. 
Nevertheless, and somewhat surprisingly, these ``pullbacks'' did 
exhibit many of the properties of pullbacks and in fact did everything 
required to support the imagined theory of symmetric lenses \cite{JR17}. 
In some sense one could ``import'' pullbacks from $\Cat$ into $\Lens$ 
by adding canonical Put operations, and the imported-pullbacks would behave 
sufficiently like real pullbacks to develop the required theory.

In our view, it is time to seriously study the categorical properties 
of the category $\Lens$. 
This paper begins that study, exploring in $\Lens$ imported pullbacks 
and real pullbacks, imported products and real products, equalisers, 
coproducts, extensivity, and a surprisingly simple proper orthogonal 
factorisation system. 
Each of these notions has important practical applications, and 
understanding the categorical nature of $\Lens$, including various 
imported exactness properties, is an important step in advancing 
applied category theory using lenses.

\subsection*{Acknowledgements}
This paper arose from the ACT2020 Adjoint School through research by the 
Maintainable Relations group. 
We are grateful to the organisers of the school for their support.
We have benefited from valuable conversations with a number of 
colleagues in the School and in our home and other institutions.
We particularly mention Chris Heunen, who asked a number of questions 
that are now answered by this paper. 
We also extend our gratitude to the anonymous referees for their 
helpful feedback on this paper. 
  
Bryce Clarke is grateful for the support of the Australian Government 
Research Training Program Scholarship.
The work of Michael Johnson is supported in part by the Australian 
Research Council.
Gioele Zardini is supported by the Swiss National Science Foundation 
under NCCR Automation, grant agreement \texttt{51NF40\_180545}, 
and he would like to thank Emilio Frazzoli for support.

\section{Background}

In this section, we recall the category $\Lens$ of small categories and 
(delta) lenses \cite{DXC11}, and establish notation for the rest of the 
paper. 
The only new result presented here is Lemma~\ref{lemma:division}(ii). 

\begin{definition}
\label{definition:delta-lens}
Let $A$ and $B$ be categories. 
A (\emph{delta}) \emph{lens} $(f, \varphi) \colon A \br B$ consists of 
a functor $f \colon A \rightarrow B$ together with a 
\emph{lifting operation}, 
\begin{equation*}
    (a \in A, u \colon fa \rightarrow b \in B) 
    \quad \longmapsto \quad 
    \varphi(a, u) \colon a \rightarrow a' \in A
\end{equation*}
which satisfies the following axioms:
\begin{enumerate}[(1)]
    \item $f\varphi(a, u) = u$
    \item $\varphi(a, 1_{fa}) = 1_{a}$
    \item $\varphi(a, v \circ u) = \varphi(a', v) \circ \varphi(a, u)$
\end{enumerate}
\end{definition}

\begin{remark}
In the literature, the functor part of a lens is often called the 
Get, while the lifting operation is called the Put. 
The three axioms are also called Put-Get, Get-Put, 
and Put-Put, respectively. 
This terminology can be confusing and distracts from the mathematics, 
so will be avoided in this paper. 
\end{remark}

\begin{example}
A \emph{split opfibration} is a lens whose chosen lifts $\varphi(a, u)$ 
are opcartesian. 
\end{example}

\begin{definition}
\label{definition:Lens-category}
Let $\Lens$ denote the category whose objects are (small) categories 
and whose morphisms are lenses. 
Given a pair of lenses 
$(f, \varphi) \colon A \br B$ and $(g, \gamma) \colon B \br C$, 
their composite is given by the functor 
$g \circ f \colon A \rightarrow C$ together the lifting operation: 
\begin{equation*}
	(a \in A, u \colon gfa \rightarrow c \in C) 
	\quad \longmapsto \quad 
	\varphi(a, \gamma(fa, u) )
\end{equation*}
The identity lens on a category $A$ consists of the identity 
functor $1_{A} \colon A \rightarrow A$ together with the trivial 
lifting operation given by projection 
$\pi(a, u \colon a \rightarrow a') = u$. 
\end{definition}

There is an identity-on-objects, forgetful functor 
$\U \colon \Lens \rightarrow \Cat$ which assigns a lens to its 
underlying functor. 
The functor $\U$ is neither \emph{full}, as not every functor can 
be given a lifting operation, nor \emph{faithful}, as a functor 
may have many possible lifting operations; however it is an 
\emph{isofibration}. 
Despite $\U$ failing to be full or faithful, 
there is a large class of functors for which there does exist a 
unique lifting operation, called discrete opfibrations, that play 
a special role in the theory of lenses. 

\begin{definition}
\label{definition:discrete-opfibration}
A functor $f \colon A \rightarrow B$ is a 
\emph{discrete opfibration} if for all pairs 
$(a \in A, u \colon fa \rightarrow b \in B)$ there exists a 
unique morphism $w \colon a \rightarrow a'$ in $A$ 
such that $fw = u$. 
A \emph{cosieve} is an injective-on-objects discrete opfibration 
(equivalently, fully faithful discrete opfibration). 
\end{definition}

Discrete opfibrations are equivalent to lenses whose lifting 
operation is an isomorphism. 
Let $\Dopf$ denote the wide subcategory of $\Cat$ whose morphisms 
are discrete opfibrations. 
Discrete opfibrations are also stable under pullback along 
arbitrary functors. 
The following result, due to Clarke \cite{Cla20b}, 
establishes the importance of discrete opfibrations for 
understanding lenses. 

\begin{proposition}
\label{proposition:triangle-rep}
Every lens $(f, \varphi) \colon A \br B$ may be represented as 
a commutative diagram of functors,
\begin{equation}
\label{equation:triangle-rep}
\begin{tikzcd}[column sep=small, row sep=small]
& X
\arrow[ld, "\varphi"']
\arrow[rd, "\phibar"]
& \\
A 
\arrow[rr, "f"'] 
& & 
B
\end{tikzcd}
\end{equation}
where $\varphi$ is a faithful, bijective-on-objects functor and 
$\phibar$ is a discrete opfibration. 
\end{proposition}

\begin{remark}
As noted in \cite{Cla20}, this result has a converse which 
implies that every lens $(f, \varphi) \colon A \br B$ may be 
identified with an equivalence class of diagrams, 
\begin{equation*}
\begin{tikzcd}[column sep=small, row sep=small]
& X
\arrow[ld, "\varphi"']
\arrow[rd, "\phibar"]
& \\
A 
\arrow[rr, "f"'] 
& & 
B
\end{tikzcd}
\qquad \simeq \qquad
\begin{tikzcd}[column sep=small, row sep=small]
& Y
\arrow[ld, "\gamma"']
\arrow[rd, "\gammabar"]
& \\
A 
\arrow[rr, "f"'] 
& & 
B
\end{tikzcd}
\end{equation*}
generated by isomorphisms $q \colon X \cong Y$ such that 
$\gamma \circ q = \varphi$ and $\gammabar \circ q = \phibar$. 
In practice, we may always identify a lens with a chosen representative 
\eqref{equation:triangle-rep} of this equivalence class. 
\end{remark}

Proposition~\ref{proposition:triangle-rep} is powerful as it allows us 
to prove results about lenses through manipulating their representation 
as diagrams in $\Cat$. 
For example, composition of lenses may be understood diagrammatically 
via pullback: 
\begin{equation}
\label{equation:composition}
\begin{tikzcd}[row sep = small, column sep = small]
& &[-1em] X \times_{B} Y
\arrow[ld]
\arrow[rd]
\arrow[dd, phantom, "\lrcorner" rotate = -45, very near start]
&[-1em] & \\
& X 
\arrow[ld, "\varphi"']
\arrow[rd, "\phibar"]
& & Y
\arrow[ld, "\gamma"']
\arrow[rd, "\overline{\gamma}"]
& \\
A
\arrow[rr, "f"']
& & B
\arrow[rr, "g"']
& & C
\end{tikzcd}
\end{equation}
This technique is central to proving many of the results in this 
paper, including the following lemma. 

\begin{lemma}
\label{lemma:division}
Consider the following commutative diagram in $\Cat$ 
with $g \colon B \rightarrow C$ a discrete opfibration: 
\begin{equation}
\label{equation:division}
\begin{tikzcd}[column sep=small, row sep=small]
A 
\arrow[rd, "g \circ f"']
\arrow[rr, "f"]
& & B
\arrow[ld, "g"]
\\
& C &
\end{tikzcd}
\end{equation}
Then: 
\begin{enumerate}[(i)]
\item If $g \circ f$ is a discrete opfibration, 
then $f$ is a discrete opfibration;
\item If $g\circ f$ has a lens structure, 
then $f$ has a unique lens structure such that 
\eqref{equation:division} commutes in $\Lens$.
\end{enumerate}
\end{lemma}
\begin{proof}
The first statement is a well-known property of discrete opfibrations. 
To prove the second statement, suppose $g \circ f$ has a lens structure 
given by the following commutative diagram of functors:
\begin{equation*}
\begin{tikzcd}[column sep=small, row sep=small]
& X
\arrow[ld, "\varphi"']
\arrow[rd, "\phibar"]
& \\
A 
\arrow[rr, "g \circ f"'] 
& & 
C
\end{tikzcd}
\end{equation*}
Now consider the commutative diagram of functors: 
\begin{equation*}
\begin{tikzcd}[column sep=small, row sep=small]
& X
\arrow[ld, "\varphi"']
\arrow[rd, "f \circ \varphi"]
& \\
A 
\arrow[rr, "f"'] 
& & 
B
\end{tikzcd}
\end{equation*}
For this to be a lens structure on $f$, we need to show that 
$f \circ \varphi$ is a discrete opfibration. 
However this follows from the first statement, 
since $g$ is a discrete opfibration and 
$g \circ (f \circ \varphi) = \phibar$ is a discrete opfibration. 
Using lens composition as in \eqref{equation:composition}, noting that 
discrete opfibrations are diagrams \eqref{equation:triangle-rep} where
$\varphi$ is an isomorphism, 
it is not difficult to show that this lens structure makes the diagram
\eqref{equation:division} commute, and that the lens structure on 
$f$ such that this holds is unique. 
\end{proof}

\section{Illustrative examples of lenses}

In this section, we present two basic examples illustrating how lenses
may arise in certain applications. 
These examples are not central to the purpose of this paper, but they may 
provide some concrete reference points for the abstract theory developed 
in the following sections.

\subsection*{State-transition machines as lenses}
\label{ex:machineinterface}

Let $B$ be a free monoid considered as a one-object category,
finitely generated by the set 
$\{\texttt{b}_1, \texttt{b}_2, \ldots, \texttt{b}_N\}$ 
where we consider the labels $\texttt{b}_{i}$ as 
\emph{interface buttons} used to operate a machine.

A lens $(f, \varphi) \colon A \br B$  can be understood as specifying 
a generalised state-transition machine, where the states are $\Ob(A)$, 
and the transitions are arrows of $A$ labelled by their domains and 
elements of the monoid $B$. 
We examine this in more detail.

The underlying functor $f$ maps arrows in $A$ to strings of labels in 
$B$. 
The lift $\varphi$ of the lens, given any object $a \in A$ and a transition 
label $\texttt{b} \in B$, selects a morphism $\varphi(a, \texttt{b})$ 
whose source is $a$.

The lifting operation $\varphi$ of the lens takes an object of $A$, 
a state of the machine, and shows what state-transition will take place if 
button $\texttt{b}_{i}$ is pressed when the machine is in that state.

In this example, the underlying functor $f$ necessarily maps all objects 
of $A$ to the single object of $B$, which suggests a natural generalisation. 
Indeed, the state-transition machine example extends to lenses with codomains 
of more than one object: the fibre of $f$ over $b \in \Ob(B)$ consists of 
a \emph{type} of states $f^{-1}(b) \subseteq \Ob(A)$, 
where the lens selects transitions out of $a \in f^{-1}(b)$ labelled by $B(b, -)$.

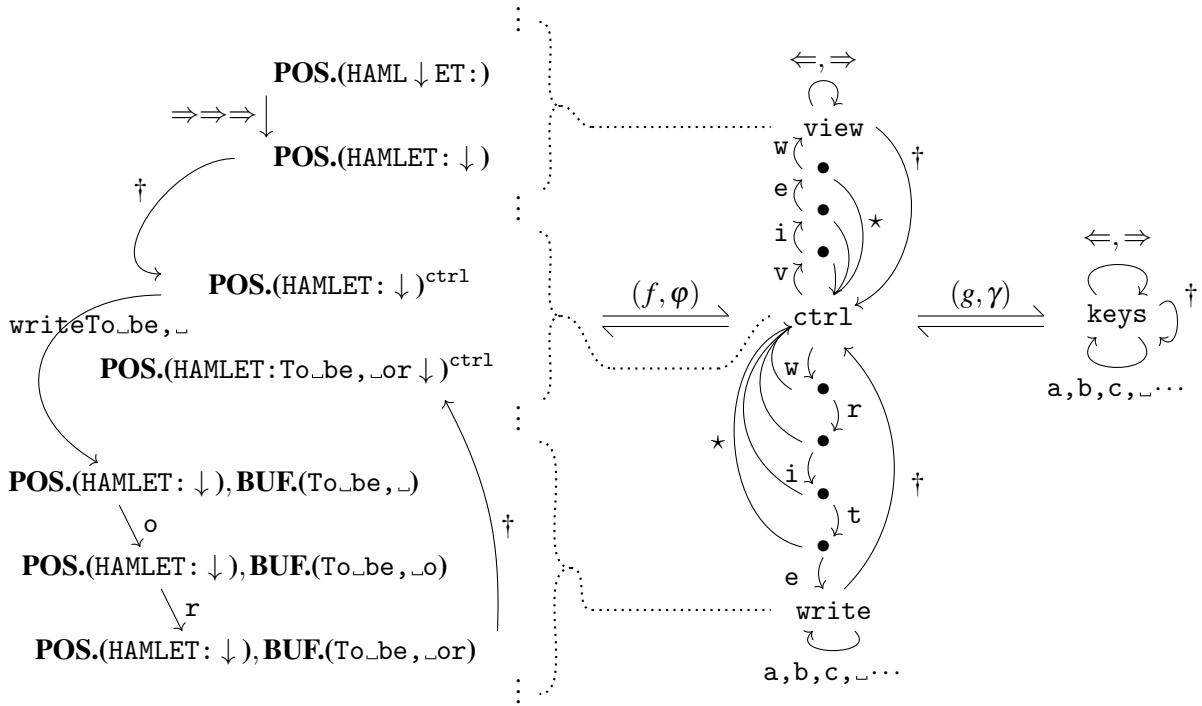
\begin{center}
\begin{scaletikzpicturetowidth}{\textwidth}
\begin{tikzpicture}[scale=\tikzscale]
	\begin{pgfonlayer}{nodelayer}
		\node [style=none] (0) at (0, 1) {$\texttt{ctrl}$};
		\node [style=none] (1) at (0.25, -6) {$\texttt{write}$};
		\node [style=none] (3) at (7, 1) {$\texttt{keys}$};
		\node [style=none] (5) at (7.5, 1.5) {};
		\node [style=none] (6) at (7.5, 0.5) {};
		\node [style=none] (7) at (6.5, 0.5) {};
		\node [style=none] (8) at (6.5, 1.5) {};
		\node [style=none] (9) at (8, 1.25) {};
		\node [style=none] (10) at (8, 0.5) {};
		\node [style=none] (14) at (0, 2.5) {$\bullet$};
		\node [style=none] (17) at (0.5, 0.25) {};
		\node [style=none] (18) at (0.5, -5.5) {};
		\node [style=none] (20) at (2.25, -3) {$\dag$};
		\node [style=none] (33) at (0, 3.5) {$\bullet$};
		\node [style=none] (34) at (-0.5, 2.5) {};
		\node [style=none] (35) at (-0.5, 3.25) {};
		\node [style=none] (36) at (0, 4.5) {$\bullet$};
		\node [style=none] (39) at (-1, 2) {$\texttt{v}$};
		\node [style=none] (40) at (-1, 3) {$\texttt{i}$};
		\node [style=none] (41) at (0.25, 5.5) {$\texttt{view}$};
		\node [style=none] (44) at (-1, 4) {$\texttt{e}$};
		\node [style=none] (45) at (-1, 5) {$\texttt{w}$};
		\node [style=none] (46) at (-0.25, 6) {};
		\node [style=none] (47) at (0.25, 6) {};
		\node [style=none] (49) at (0, -0.75) {$\bullet$};
		\node [style=none] (50) at (-0.25, 0.25) {};
		\node [style=none] (51) at (-0.25, -0.5) {};
		\node [style=none] (52) at (-0.75, -0.25) {$\texttt{w}$};
		\node [style=none] (58) at (0.25, -1) {};
		\node [style=none] (59) at (0.25, -1.75) {};
		\node [style=none] (60) at (0.75, -1.25) {$\texttt{r}$};
		\node [style=none] (61) at (0, -2) {$\bullet$};
		\node [style=none] (63) at (0, -3.25) {$\bullet$};
		\node [style=none] (64) at (-0.25, -2.25) {};
		\node [style=none] (65) at (-0.25, -3) {};
		\node [style=none] (66) at (-0.75, -2.75) {$\texttt{i}$};
		\node [style=none] (67) at (0.25, -3.5) {};
		\node [style=none] (68) at (0.25, -4.25) {};
		\node [style=none] (70) at (0, -4.5) {$\bullet$};
		\node [style=none] (72) at (0, -4.75) {};
		\node [style=none] (73) at (0, -5.5) {};
		\node [style=none] (74) at (-0.75, -5.25) {$\texttt{e}$};
		\node [style=none] (75) at (0.75, -3.75) {$\texttt{t}$};
		\node [style=none] (76) at (-0.5, -4.5) {};
		\node [style=none] (77) at (-0.75, 0.75) {};
		\node [style=none] (78) at (-0.5, -3.25) {};
		\node [style=none] (79) at (-0.5, -2) {};
		\node [style=none] (80) at (-0.75, -0.75) {};
		\node [style=none] (84) at (-2.5, -2) {$\star$};
		\node [style=none] (85) at (0.25, 2.25) {};
		\node [style=none] (86) at (0.25, 1.5) {};
		\node [style=none] (87) at (0.25, 3.25) {};
		\node [style=none] (88) at (0.25, 4.25) {};
		\node [style=none] (90) at (1.25, 3.25) {$\star$};
		\node [style=none] (96) at (2.25, 1) {};
		\node [style=none] (97) at (5.25, 1) {};
		\node [style=none] (98) at (2.25, 0.75) {};
		\node [style=none] (99) at (5.25, 0.75) {};
		\node [style=none] (100) at (5, 1.25) {};
		\node [style=none] (101) at (2.5, 0.5) {};
		\node [style=none] (102) at (-5.25, 1) {};
		\node [style=none] (103) at (-2.25, 1) {};
		\node [style=none] (104) at (-5.25, 0.75) {};
		\node [style=none] (105) at (-2.25, 0.75) {};
		\node [style=none] (106) at (-2.5, 1.25) {};
		\node [style=none] (107) at (-5, 0.5) {};
		\node [style=none] (108) at (8.75, 1.5) {$\dag$};
		\node [style=none] (109) at (7, -0.75) {$\texttt{a,b,c,\textvisiblespace}\cdots$};
		\node [style=none] (110) at (7, 2.75) {$\Leftarrow,\Rightarrow$};
		\node [style=none] (111) at (-13.5, -7) {$\textbf{POS.(}\texttt{HAMLET:}\downarrow\textbf{)},\textbf{BUF.(}\texttt{To\textvisiblespace be,\textvisiblespace or}\textbf{)}$};
		\node [style=none] (112) at (-14, -5) {$\textbf{POS.(}\texttt{HAMLET:}\downarrow\textbf{)},\textbf{BUF.(}\texttt{To\textvisiblespace be,\textvisiblespace o}\textbf{)}$};
		\node [style=none] (113) at (-14.5, -3) {$\textbf{POS.(}\texttt{HAMLET:}\downarrow\textbf{)},\textbf{BUF.(}\texttt{To\textvisiblespace be,\textvisiblespace}\textbf{)}$};
		\node [style=none] (114) at (-16.75, -3.5) {};
		\node [style=none] (115) at (-16.25, -4.5) {};
		\node [style=none] (116) at (-15.75, -5.5) {};
		\node [style=none] (117) at (-15.25, -6.5) {};
		\node [style=none] (118) at (-16, -4) {$\texttt{o}$};
		\node [style=none] (119) at (-15, -6) {$\texttt{r}$};
		\node [style=none] (121) at (-7.25, -7.75) {$\vdots$};
		\node [style=none] (128) at (0.75, -6.5) {};
		\node [style=none] (129) at (-0.25, -6.5) {};
		\node [style=none] (137) at (0, 7) {$\Leftarrow,\Rightarrow$};
		\node [style=none] (139) at (0.25, -7.5) {$\texttt{a,b,c,\textvisiblespace}\cdots$};
		\node [style=none] (141) at (-10.5, 6.75) {$\textbf{POS.(}\texttt{HAML}\downarrow\texttt{ET:}\textbf{)}$};
		\node [style=none] (142) at (-13.25, 6.25) {};
		\node [style=none] (143) at (-13.25, 5.25) {};
		\node [style=none] (144) at (-14.5, 5.75) {$\Rightarrow \Rightarrow \Rightarrow$};
		\node [style=none] (145) at (-10.5, 4.75) {$\textbf{POS.(}\texttt{HAMLET:}\downarrow\textbf{)}$};
		\node [style=none] (146) at (-7.25, 8.25) {$\vdots$};
		\node [style=none] (147) at (-7.25, 3.75) {$\vdots$};
		\node [style=none] (148) at (-6.75, 8) {};
		\node [style=none] (149) at (-6.75, 4) {};
		\node [style=none] (150) at (-6.25, 6) {};
		\node [style=none] (151) at (-1.25, 5.5) {};
		\node [style=none] (152) at (-5.5, 5.5) {};
		\node [style=none] (153) at (-3.75, 1.5) {$(f,\varphi)$};
		\node [style=none] (154) at (3.75, 1.5) {$(g,\gamma)$};
		\node [style=none] (155) at (-0.5, 1.5) {};
		\node [style=none] (156) at (-0.5, 2.25) {};
		\node [style=none] (157) at (-0.5, 3.5) {};
		\node [style=none] (158) at (-0.5, 4.25) {};
		\node [style=none] (159) at (-0.5, 4.5) {};
		\node [style=none] (160) at (-0.5, 5.25) {};
		\node [style=none] (161) at (1.25, 5.5) {};
		\node [style=none] (162) at (0.75, 1.25) {};
		\node [style=none] (163) at (2.25, 4.75) {$\dag$};
		\node [style=none] (164) at (-11.5, 1.75) {$\textbf{POS.(}\texttt{HAMLET:}\downarrow\textbf{)}^{\texttt{ctrl}}$};
		\node [style=none] (168) at (-12.5, -0.25) {$\textbf{POS.(}\texttt{HAMLET:To\textvisiblespace be,\textvisiblespace or}\downarrow\textbf{)}^{\texttt{ctrl}}$};
		\node [style=none] (170) at (-7.25, -1.25) {$\vdots$};
		\node [style=none] (171) at (-6.75, 3) {};
		\node [style=none] (172) at (-6.75, -1) {};
		\node [style=none] (173) at (-6.25, 1) {};
		\node [style=none] (174) at (-1.25, 1) {};
		\node [style=none] (175) at (-5, -0.25) {};
		\node [style=none] (176) at (-2.5, -0.25) {};
		\node [style=none] (177) at (-6.75, -2) {};
		\node [style=none] (178) at (-6.75, -8) {};
		\node [style=none] (179) at (-6, -5) {};
		\node [style=none] (180) at (-1.25, -6) {};
		\node [style=none] (181) at (-5.5, -6) {};
		\node [style=none] (183) at (-15.5, 0) {};
		\node [style=none] (184) at (-14, 4.75) {};
		\node [style=none] (185) at (-15.75, 2) {};
		\node [style=none] (186) at (-16.25, 4) {$\dag$};
		\node [style=none] (188) at (-15.75, 1.5) {};
		\node [style=none] (189) at (-17.25, -2.5) {};
		\node [style=none] (190) at (-17.25, 0.75) {$\texttt{writeTo\textvisiblespace be,\textvisiblespace}$};
		\node [style=none] (191) at (-7.75, -6.5) {};
		\node [style=none] (192) at (-9, -1) {};
		\node [style=none] (193) at (-7.5, -4) {$\dag$};
	\end{pgfonlayer}
	\begin{pgfonlayer}{edgelayer}
		\draw [style=arrow, in=45, out=135, looseness=3.50] (8.center) to (5.center);
		\draw [style=arrow, in=-135, out=-45, looseness=3.50] (6.center) to (7.center);
		\draw [style=arrow, bend left=135, looseness=2.75] (9.center) to (10.center);
		\draw [style=arrow, bend right=45] (18.center) to (17.center);
		\draw [style=arrow, bend left=45, looseness=1.25] (34.center) to (35.center);
		\draw [style=arrow, bend left=120, looseness=4.75] (46.center) to (47.center);
		\draw [style=arrow, bend right] (50.center) to (51.center);
		\draw [style=arrow, bend left] (58.center) to (59.center);
		\draw [style=arrow, bend right] (64.center) to (65.center);
		\draw [style=arrow, bend left] (67.center) to (68.center);
		\draw [style=arrow, bend right] (72.center) to (73.center);
		\draw [style=arrow, bend left=75] (76.center) to (77.center);
		\draw [bend left=60, looseness=1.25] (78.center) to (77.center);
		\draw [bend left=60, looseness=1.25] (79.center) to (77.center);
		\draw [bend left=60, looseness=1.25] (80.center) to (77.center);
		\draw [style=arrow, bend right=345, looseness=0.75] (85.center) to (86.center);
		\draw [bend left=45] (87.center) to (86.center);
		\draw [bend left=60] (88.center) to (86.center);
		\draw (96.center) to (97.center);
		\draw (99.center) to (98.center);
		\draw (100.center) to (97.center);
		\draw (98.center) to (101.center);
		\draw (102.center) to (103.center);
		\draw (105.center) to (104.center);
		\draw (106.center) to (103.center);
		\draw (104.center) to (107.center);
		\draw [style=arrow] (114.center) to (115.center);
		\draw [style=arrow] (116.center) to (117.center);
		\draw [style=arrow, in=-135, out=-45, looseness=2.25] (128.center) to (129.center);
		\draw [style=arrow] (142.center) to (143.center);
		\draw [style=arrow, bend left=45, looseness=1.25] (155.center) to (156.center);
		\draw [style=arrow, bend left=45, looseness=1.25] (157.center) to (158.center);
		\draw [style=arrow, bend left=45, looseness=1.25] (159.center) to (160.center);
		\draw [style=arrow, bend left=60] (161.center) to (162.center);
		\draw [style=reddishdashed, in=180, out=0] (150.center) to (152.center);
		\draw [style=reddishdashed] (152.center) to (151.center);
		\draw [style=reddishdashed, in=180, out=0, looseness=0.75] (148.center) to (150.center);
		\draw [style=reddishdashed, in=0, out=-180, looseness=0.75] (150.center) to (149.center);
		\draw [style=reddishdashed, in=180, out=0] (173.center) to (175.center);
		\draw [style=reddishdashed, in=180, out=0, looseness=0.75] (171.center) to (173.center);
		\draw [style=reddishdashed, in=0, out=-180, looseness=0.75] (173.center) to (172.center);
		\draw [style=reddishdashed] (175.center) to (176.center);
		\draw [style=reddishdashed, in=-165, out=0, looseness=0.75] (176.center) to (174.center);
		\draw [style=reddishdashed, in=180, out=0] (179.center) to (181.center);
		\draw [style=reddishdashed] (181.center) to (180.center);
		\draw [style=reddishdashed, in=180, out=0, looseness=0.75] (177.center) to (179.center);
		\draw [style=reddishdashed, in=0, out=-180, looseness=0.75] (179.center) to (178.center);
		\draw [style=arrow, in=-180, out=180, looseness=1.25] (184.center) to (185.center);
		\draw [style=arrow, in=150, out=-180, looseness=1.75] (188.center) to (189.center);
		\draw [style=arrow, bend right=15] (191.center) to (192.center);
	\end{pgfonlayer}
\end{tikzpicture}

\end{scaletikzpicturetowidth}
\end{center}

\begin{example}[``typed'' state-transition machines, and composition of lenses]
We sketch a rudimentary text-editor program operated by keystrokes from 
a keyboard. 
The $\mathit{STATE}$ category where objects are internal states of the program 
might resemble the leftmost diagram above: 
objects are tuples of strings with marked ($\downarrow$) cursor 
\textbf{positions} modelling text files, 
along with text \textbf{buffers} that hold onto strings of text to be inserted. 
We depict the path starting from the 
$\textbf{POS.(}\texttt{HAML}\downarrow\texttt{ET:}\textbf{)}$ state in 
\texttt{view}-mode, and inputting the keyboard sequence 
$\Rightarrow\Rightarrow\Rightarrow\dag\texttt{writeTo\textvisiblespace be,\textvisiblespace or}\dag$.

The program may have \emph{modes of operation}, 
such that the same key on the keyboard has different functions depending 
on the current mode of operation. 
We depict the $\mathit{MODE}$ category in the middle. 
In \texttt{view}-mode, arrow keys move the cursor's position through text. 
The special key $\dag$ enters \texttt{control}-mode which keeps memories 
of cursor position intact, while awaiting strings \texttt{view} or 
\texttt{write} to switch to another mode; failed commands return to \texttt{control}-mode, notated by wildcard $\star$ arrows in the diagram. 
The \texttt{write}-mode allows alphabetic inputs to fill a temporary text buffer, 
the contents of which are appended to the main body of text upon returning 
to \texttt{control}-mode. 
We model the coordination between $\mathit{STATE}$ and $\mathit{MODE}$ as a lens, 
in fact a discrete opfibration, $(f,\varphi) \colon \mathit{STATE} \br \mathit{MODE}$. 
The ``typing'' of states by modes arises from the fact that the fibre 
of $f$ over $\texttt{write}$ contains all states of the program 
accessible in \texttt{write}-mode, and similarly for the fibres of $f$ 
above $\texttt{ctrl}$ and $\texttt{view}$.

We model the $\mathit{KEYBOARD}$ as a one-object category with generating 
endomorphisms of alphabetic keys $\texttt{a,b,c,\textvisiblespace}\ldots$, 
arrow keys $\Leftarrow,\Rightarrow$ for navigation, and a command key $\dag$. 
The $\mathit{MODE}$ category is a state-machine over $\mathit{KEYBOARD}$, so we coordinate the 
two with a lens $(g,\gamma) \colon \mathit{MODE} \br \mathit{KEYBOARD}$. 
Altogether, we have a composition of lenses between categories 
$\mathit{STATE} \br \mathit{MODE} \br \mathit{KEYBOARD}$. 
\end{example}

\subsection*{Collaborative design strategies as lenses}

\vspace{-0.9cm}
\begin{figure}[tbh]
\begin{subfigure}[b]{0.45\linewidth}
\begin{center}
\begin{tikzpicture}[scale=0.85,transform shape]
\node at (0,0){
\begin{tikzcd}[row sep=small]
\fast & 
\\
\average
\arrow[u, dash] 
& \expensive 
\\
\slow 
\arrow[u, dash] 
& \cheap
\arrow[u, dash]
\\[-10pt]
F & R
\end{tikzcd}};
\end{tikzpicture}
\caption{Functionalities and resources.\label{fig:fun-res}}
\end{center}
\end{subfigure}
\begin{subfigure}[b]{0.45\linewidth}
\begin{center}
\begin{tikzpicture}[scale=0.85,transform shape]
\draw[dotted, thick, rounded corners=5mm]  (-3.6,-0.5) rectangle ++(2.7cm, 1.75cm);
\draw[dotted, thick, rounded corners=10mm]  (0.25,-1.5) rectangle ++(3.3cm, 3.9cm);
\node at (0,0){
\begin{tikzcd}[row sep=small]
& (\fast,\expensive) 
\arrow{d} 
\\
(\fast,\cheap) 
\arrow{ur}
\arrow{d}
& (\average, \expensive) 
\arrow{d} 
\\
(\average,\cheap) 
\arrow{dr}
\arrow{ur}
\arrow[dd, dash, dotted]
& (\slow,\expensive) 
\\
& (\slow,\cheap) 
\arrow{u}
\arrow[d, dash, dotted]
\\
\false 
\arrow{r}
& \true 
\end{tikzcd}};
\end{tikzpicture}
\caption{Fibre representation of a boolean profunctor.\label{fig:fibre}}
\end{center}
\end{subfigure}
\end{figure}

The monotone theory of co-design presented in 
\cite{Cen16,FS19} has found concrete 
applications in engineering, ranging from the design of intermodal 
mobility systems \cite{ZLSCFP20} to robotics and control 
\cite{ZCF21, ZMCF21}.

Let $F$ be a poset representing \emph{functionalities}, 
let $R$ be a poset representing \emph{costs} or \emph{resources}, 
and let $\Bool$ be the two element poset $\{ \false \rightarrow \true \}$. 
A \emph{boolean profunctor}, denoted by $F \nrightarrow R$, is a 
functor $F^{\op} \times R \rightarrow \Bool$ 
which captures a relation between functionalities and 
requirements modelling feasibility, where decreasing demanded functionalities, 
or increasing resources, both increase feasibility.

Consider hiring an autonomous vehicle (AV): 
depending on how sophisticated the AV will be, the ride cost might change. 
Suppose $F$ is the poset of performance grades of the AV, and $R$ is the 
poset of ride costs (see (a) above). 
We define a boolean profunctor relating $F$ and $R$ following the 
rationale that the only cheap rides are slow rides, and to get average 
and fast rides one needs to pay more.

Objectwise, a boolean profunctor behaves as a judgement of whether 
each $(F, R)$ pair is feasible, which is evident when we view the 
functor fibre-wise over $\Bool$ (see (b) above).
A lens structure on such a functor additionally provides, 
for each infeasible pair, a specified (reachable) feasible $(F, R)$ pair. 
For instance, the pair $(\average,\cheap)$ is infeasible. 
Possible ways to get feasible scenarios include accepting paying more 
(i.e.~mapping to $(\average,\expensive)$) or sacrificing performance 
(i.e.~mapping to $(\slow,\cheap)$). 
The lifting operation of a lens structure chooses one alternative.

Altogether, a lens in this setting models someone's design opinion: 
whether or not something is feasible, along with a 
\emph{satisfaction strategy} that informs how to concretely 
compromise infeasible parameters, by either lowering demanded 
functionalities or increasing supplied resources. 

\section{Limits, colimits, and a factorisation system}

In this section, we show that the category $\Lens$ has a terminal object, 
an initial object, small coproducts, and equalisers. 
We also provide a characterisation of the monomorphisms and epimorphisms, 
and prove that $\Lens$ has an (epi, mono)-factorisation system. 

\begin{proposition}
\label{proposition:terminal-object}
The category $\Lens$ has a terminal object.
\end{proposition}
\begin{proof}
The terminal object in $\Lens$, as in $\Cat$, is the discrete category 
$1$ with a single object.
Given a category $A$, the unique lens $A \br 1$ consists of the unique 
functor $! \colon A \rightarrow 1$ together with the 
trivial lifting operation.
Following Proposition~\ref{proposition:triangle-rep}, this lens may be 
represented as the commutative diagram,
\begin{equation}
\label{equation:terminal-object}
\begin{tikzcd}[column sep=small, row sep=small]
& A_{0}
\arrow[ld, "i"']
\arrow[rd, "!"]
& \\
A 
\arrow[rr, "!"'] 
& & 
1
\end{tikzcd}
\end{equation}
where $i \colon A_{0} \rightarrow A$ is the inclusion of the discrete 
category $A_{0}$ of objects into $A$. 
\end{proof}

\begin{example}[The terminal interface]
The terminal object $1$ in this setting is an interface with a single 
button (the identity) which does nothing. 
The lift of an identity is an identity, so pressing the button does 
not change the state of the machine. 
All machines are compatible with a `do-nothing' interface.
\end{example}

\begin{proposition}
\label{proposition:initial-object}
The category $\Lens$ has an initial object.
\end{proposition}
\begin{proof}
The initial object in $\Lens$, as in $\Cat$, is the empty category $0$. 
Given a category $A$, the unique lens $0 \br A$ consists of the unique 
functor $! \colon 0 \rightarrow A$ together with the 
trivial lifting operation. 
\begin{equation}
\label{equation:initial-object}
\begin{tikzcd}[column sep=small, row sep=small]
& 0
\arrow[ld, equal]
\arrow[rd, "!"]
& \\
0
\arrow[rr, "!"'] 
& & 
A
\end{tikzcd}
\end{equation}
Following Proposition~\ref{proposition:triangle-rep}, this lens may be 
represented as the commutative diagram above.
\end{proof}

\begin{example}[The initial machine]
The initial object $0$ in this setting is the null machine with no 
internal states, which is compatible with any (unplugged) keyboard $A$.
\end{example}

\begin{proposition}
\label{proposition:coproducts}
The category $\Lens$ has small coproducts. 
\end{proposition}
\begin{proof}
Given a pair of categories $A$ and $B$, their coproduct $A + B$
in $\Lens$ coincides with their coproduct in $\Cat$. 
The coproduct injections in $\Cat$ are discrete opfibrations, and 
therefore have a unique lens structure. 
To see that the universal property holds, consider a pair of lenses
$(f, \varphi) \colon A \br B$ and $(g, \gamma) \colon C \br B$ 
represented as commutative diagrams following 
Proposition~\ref{proposition:triangle-rep}: 
\begin{equation*}
\begin{tikzcd}[column sep=small, row sep=small]
& X
\arrow[ld, "\varphi"']
\arrow[rd, "\phibar"]
& \\
A 
\arrow[rr, "f"'] 
& & 
B
\end{tikzcd}
\qquad \qquad
\begin{tikzcd}[column sep=small, row sep=small]
& Y
\arrow[ld, "\gamma"']
\arrow[rd, "\gammabar"]
& \\
C 
\arrow[rr, "g"'] 
& & 
B
\end{tikzcd}
\end{equation*}
Since bijective-on-objects functors are closed under coproducts, and 
$\Dopf / B$ has coproducts, the unique lens $A + C \br B$ 
is represented by the commutative diagram: 
\begin{equation}
\label{equation:coproduct}
\begin{tikzcd}[column sep=small, row sep=small]
&[-1em] X + Y
\arrow[ld, "\varphi + \gamma"'] 
\arrow[rd, "{[\phibar,\, \gammabar]}"]
& \\
A + C
\arrow[rr, "{[f,\, g]}"']
& & B
\end{tikzcd}
\end{equation}
The above arguments extend to coproducts indexed by any set. 
\end{proof}

\begin{example}[Coproduct interfaces]
Consider $A$ and $C$ to be windowed programs that operate through a 
common interface $B$, a keyboard. 
The coproduct machine $A + C$ behaves as a window manager, that focuses 
on one window: functionally, the window manager forwards keystrokes 
from $B$ to whichever of $A$ or $C$ is currently in focus.
\end{example}

Unlike the previous examples of limits and colimits, 
equalisers in $\Lens$ are an example which does not coincide 
with the equaliser of the underlying functors in $\Cat$. 

\begin{proposition}
\label{proposition:equalisers}
The category $\Lens$ has equalisers.
\end{proposition}
\begin{proof}[Proof (sketch)]
Consider a parallel pair of lenses $(f, \varphi) \colon A \br B$ 
and $(g, \gamma) \colon A \br B$, and construct the equaliser 
$j \colon E \rightarrow A$ of their underlying functors in $\Cat$. 
The equaliser of the parallel pair of lenses is the largest subobject 
$m \colon M \rightarrowtail E$ such that $j \circ m \colon M \rightarrow A$ 
is a discrete opfibration which forms a cone over the parallel pair 
in $\Lens$. 
\end{proof}

\begin{example}[Equalising co-design strategies]
Consider a parallel pair of lenses 
$(f, \varphi)\colon F^{\op} \times R \br \Bool$ and 
$(g, \gamma) \colon F^{\op} \times R \br \Bool$ 
to model two experts' opinions on the design problem encoded 
by $F^{\op} \times R$.
Their equaliser $E \br F^{\op} \times R$ is an embedding of 
$E$ into $F^{\op} \times R$, which selects all 
pairs in $F^{\op} \times R$ such that the feasibility 
judgements $f$ and $g$ agree, and moreover, such that the satisfaction 
strategies $\varphi$ and $\gamma$ concur. 
The equaliser always exists: in the worst case where there is total 
disagreement, $E = 0$.
\end{example}
 
\begin{corollary}
\label{corollary:split-idempotents}
In the category $\Lens$, all idempotents split. 
\end{corollary}
\begin{proof}
The splitting of an idempotent lens is given by the
equaliser with the identity lens. 
\end{proof}

\begin{remark}
Split idempotents are simple kinds of limits, but are interesting here 
for two reasons: they are also examples of coequalisers in $\Lens$ 
(which are explored further in the paper by Di Meglio \cite{MattACT21}) and 
they are also absolute (co)limits, meaning that they are examples of 
(co)equalisers which are preserved by any functor, in particular, by
the forgetful functor $\U \colon \Lens \rightarrow \Cat$. 
\end{remark}

Both coproduct injections and equalisers are examples of monomorphisms
in $\Lens$. 
We now turn our attention to establishing sufficient conditions for
a lens to be a monomorphism or an epimorphism. 

\begin{lemma}
\label{lemma:mono-lens}
If a lens is an injective-on-objects discrete opfibration, 
then it is a monomorphism. 
\end{lemma} 
\begin{proof}
Every injective-on-objects discrete opfibration is also 
injective-on-morphisms, thus a monomorphism in $\Cat$. 
Consider the following diagram in $\Lens$ 
(which omits the information of the lifting operation), 
consisting of a parallel 
pair of lenses $f$ and $f'$ which are equal to a lens $h$ under 
post-composition by an injective-on-objects discrete opfibration $g$: 
\begin{equation*}
\begin{tikzcd}[row sep=small, column sep=small]
A 
\arrow[rr, shift left, "f"]
\arrow[rr, shift right, "f'"']
\arrow[rd, "h"']
& & B
\arrow[ld, "g"]
\\
& C &
\end{tikzcd}
\end{equation*}
Since $g$ is a monomorphism in $\Cat$, the underlying functors of $f$ 
and $f'$ are equal. 
Furthermore, by Lemma~\ref{lemma:division}, the lifting operations on 
$f$ and $f'$ are also equal. 
\end{proof}

\begin{proposition}
\label{proposition:reflect-monos}
The functor $\U \colon \Lens \rightarrow \Cat$ reflects monomorphisms.
\end{proposition}
\begin{proof}
We need to show that if a lens $(f, \varphi) \colon A \br B$ has an
underlying functor $f$ which is a monomorphism in $\Cat$, 
then the lens is a monomorphism. 
Since such a lens is injective-on-objects, by 
Lemma~\ref{lemma:mono-lens} it suffices to show that it is also a 
discrete opfibration. 
Now for each pair $(a \in A, u \colon fa \rightarrow b \in B)$, 
there exists a unique morphism $\varphi(a, u)$ in $A$ such that 
$f\varphi(a, u) = u$, since $f$ is injective-on-morphisms. 
\end{proof}

\begin{lemma}
\label{lemma:epi-lens}
If a lens is surjective-on-objects, then it is an epimorphism. 
\end{lemma}
\begin{proof}
Consider a surjective-on-objects lens 
$(f, \varphi) \colon A \br B$. 
Then $(f, \varphi)$ must also be surjective-on-morphisms, since given any 
morphism $u \colon b \rightarrow b'$ in $B$, there exists an object 
$a$ such that $fa = b$, and thus from the lifting operation a morphism 
$\varphi(a, u) \colon a \rightarrow a'$ in $A$ such that 
$f\varphi(a, u) = u$. 
Therefore the underlying functor $f \colon A \rightarrow B$ is an 
epimorphism in $\Cat$. 
Now consider a parallel pair of lenses $(g, \gamma) \colon B \br C$ 
and $(g', \gamma') \colon B \br C$ such that $g \circ f = g' \circ f$ 
and $\varphi(a, \gamma(fa, u)) = \varphi(a, \gamma'(fa, u))$
for all pairs $(a \in A, u \colon gfa \rightarrow c \in C)$. 
Then $g = g'$ since $f$ is an epimorphism, and 
$\gamma(fa, u) = \gamma'(fa, u)$ since they are both equal to 
$f\varphi(a, \gamma(fa, u))$.
\end{proof}

\begin{corollary}
\label{corollary:reflect-epics}
The functor $\U \colon \Lens \rightarrow \Cat$ reflects epimorphisms.
\end{corollary}
\begin{proof}
This follows from Lemma~\ref{lemma:epi-lens}, since
every epimorphism in $\Cat$ is surjective-on-objects. 
\end{proof}

While Lemma~\ref{lemma:mono-lens} and Lemma~\ref{lemma:epi-lens} only 
provide sufficient conditions for monomorphisms and epimorphisms in 
$\Lens$, it is natural to wonder if they are also necessary conditions. 
This is indeed the case and is proved by Di Meglio \cite{MattACT21}.
Altogether, these results provide the following characterisation of 
monomorphisms and epimorphisms in $\Lens$. 

\begin{proposition}
\label{proposition:lens-mono}
A lens $(f, \varphi) \colon A \br B$ is a monomorphism if and only if 
any of the following hold: 
\begin{enumerate}
\item $(f, \varphi)$ is an injective-on-objects discrete opfibration;
\item $(f, \varphi)$ is a fully faithful discrete opfibration;
\item $f$ is a monomorphism in $\Cat$. 
\end{enumerate}
\end{proposition}

\begin{proposition}
\label{proposition:lens-epi}
A lens $(f, \varphi) \colon A \br B$ is an epimorphism if and only if 
any of the following hold: 
\begin{enumerate}
\item $f$ is surjective-on-objects;
\item $f$ is surjective-on-morphisms.
\end{enumerate}
\end{proposition}

It is surprising that unlike $\Cat$, the epimorphisms in $\Lens$ admit 
a simple characterisation; epimorphisms in $\Lens$ are discussed 
further in \cite{MattACT21}. 
Together, Proposition~\ref{proposition:lens-mono} and 
Proposition~\ref{proposition:lens-epi} have several consequences, 
including that $\Lens$ is a \emph{balanced} category. 

\begin{corollary}
\label{corollary:balanced}
A lens is an isomorphism if and only if it is a monomorphism and an 
epimorphism. 
\end{corollary}
\begin{proof}
It is immediate that every bijective-on-objects 
(that is, both injective-on-objects and surjective-on-objects)
discrete opfibration is an isomorphism, and conversely. 
\end{proof}

In a recent paper by Johnson and Rosebrugh \cite{JR21}, it was noted that 
$\Lens$ admits a proper orthogonal factorisaton system. 
Using the above propositions this is actually an 
(epi,~mono)-factorisation system, meaning that the left class is
exactly the epimorphisms, and the right class is exactly the 
monomorphisms.
We now provide a (new) proof of this result based on the following
two known results.

\begin{lemmaA}
There is an orthogonal factorisation system on $\Cat$ which factors
every functor into a surjective-on-objects functor followed by an
injective-on-objects fully faithful functor. 
\end{lemmaA}

\begin{lemmaB}
There is an (epi, mono)-factorisation system on $\Dopf$ which factors
every discrete opfibration into a surjective-on-objects discrete 
opfibration (epimorphism) followed by an injective-on-objects 
discrete opfibration (monomorphism). 
\end{lemmaB}

Note that the second lemma is a special case of the first, in the 
sense that the canonical inclusion functor $\Dopf \rightarrow \Cat$ 
preserves the factorisation system. 
We are now able to prove the following result. 

\begin{theorem}
\label{theorem:factorisation}
The category $\Lens$ has an orthogonal factorisation system which factors
every lens into a surjective-on-objects lens (epimorphism) followed by 
a cosieve (monomorphism). 
\end{theorem}
\begin{proof}
Consider a lens $(f, \varphi) \colon A \br B$ represented by the 
diagram \eqref{equation:triangle-rep}. 
By Lemma~B, we can factorise $\phibar \colon X \rightarrow B$ into 
a surjective-on-objects discrete opfibration $j \colon X \rightarrow I$ 
followed by an injective-on-objects (fully faithful) discrete opfibration
$k \colon I \rightarrow B$. 
By Lemma~A, the orthogonality property induces a unique 
functor $f'$ which is necessarily surjective-on-objects:
\begin{equation*}
\begin{tikzcd}
X 
\arrow[d, "\varphi"']
\arrow[r, "j"]
& I
\arrow[d, "k"]
\\
A 
\arrow[r, "f"']
\arrow[ru, dashed, "f'"]
& B
\end{tikzcd}
\end{equation*}
This provides the (epi, mono)-factorisation of the lens 
$(f, \varphi) \colon A \br B$ as claimed. 

To show this is an orthogonal factorisation system, consider the 
following diagram in $\Lens$ where $e$ is an epimorphism and 
$m$ is a monomorphism: 
\begin{equation}
\label{equation:orthogonality}
\begin{tikzcd}
A
\arrow[d, two heads, "e"']
\arrow[r, "f"]
& C
\arrow[d, tail, "m"]
\\
B 
\arrow[r, "g"']
& D
\end{tikzcd}
\end{equation}
Considering the diagram \eqref{equation:orthogonality} under the 
forgetful functor $\Lens \rightarrow \Cat$, 
by Lemma~A there exists a unique functor 
$h \colon B \rightarrow C$ such that $h \circ e = f$ and $m \circ h = g$
in $\Cat$. 
Since $m$ is a discrete opfibration, by Lemma~\ref{lemma:division} 
the functor $h$ has a unique lens structure such that 
$m \circ h = g$ in $\Lens$. 
Moreover, since $m$ is a monomorphism in $\Lens$, we also have 
that $h \circ e = f$ in $\Lens$. 
This proves the claim of orthogonality. 
\end{proof}

\begin{remark}
It is interesting to note that the forgetful functor 
$\Lens \rightarrow \Cat$ sends the (epi, mono)-factorisation in $\Lens$
to both the orthogonal factorisation system on $\Cat$ stated in 
Lemma~A, as well as the classical \emph{image factorisation} of a 
functor. 
\end{remark}
\begin{example}[BIOS / OS factorisation]
Recall that when interpreting lenses as state machines, 
the objects in the codomain of the lens can model \emph{modes} or \emph{types} 
of states in the domain. 
For a computer, such a codomain might look like the two-object category
$\{ \texttt{BIOS} \rightarrow \texttt{OS} \}$ with some additional 
endomorphisms. 
The arrow models the fact that the BIOS is encountered at startup, and 
if nothing is done to stay in the BIOS, there is a one-way transition 
into the OS where all everyday operations occur. 

A software engineer who is only interested in the everyday operations 
is concerned only with the behaviour of the computer over the 
$\texttt{OS}$ states. 
This leads to a factorisation of 
$\{ \texttt{EverydayOperation} \} \br \{ \texttt{BIOS} \rightarrow \texttt{OS} \}$ 
as the epimorphism of interest 
$\{ \texttt{EverydayOperation} \} \br \{ \texttt{OS} \}$, followed by the 
embedding monomorphism 
$\{ \texttt{OS} \} \br \{ \texttt{BIOS} \rightarrow \texttt{OS} \}$.
\end{example}

\section{Imported limits, distributivity, and extensivity}

In this section, we introduce a notion of \emph{imported limits}, 
and show that the category $\Lens$ has imported products and imported 
pullbacks. 
While generally imported limits do not coincide with limits in $\Lens$,
we show that $\Lens$ admits all products with discrete categories, 
and all pullbacks along discrete opfibrations. 
We also show that imported products and imported pullbacks in $\Lens$ 
behave nicely with coproducts, proving that $\Lens$ is a 
distributive and extensive category. 
 
\begin{definition}
\label{definition:imported-limit}
The \emph{imported limit} of a diagram $\D \colon J \rightarrow \Lens$
along the forgetful functor $\U \colon \Lens \rightarrow \Cat$ 
is a canonical cone $\Delta_{\D}$ over $\D$ 
such that $\U \circ \Delta_{\D}$ coincides with the limit of the diagram
$\U \circ \D \colon J \rightarrow \Cat$. 
\end{definition}

\begin{remark}
The above definition is an attempt to describe the phenomenon where 
the projection functors from a limit in $\Cat$ 
(for example, products or pullbacks) have canonical lens 
structures, without explaining what is meant by \emph{canonical}. 
A thorough investigation of this concept is planned for future work. 
\end{remark}

Every limit \emph{created} by the forgetful functor 
$\U \colon \Lens \rightarrow \Cat$ is an imported limit; 
for example, terminal objects and monomorphisms. 
The goal of this section is to consider two examples of 
imported limits which are \emph{not necessarily} limits in $\Lens$. 

\subsection*{Imported products and distributivity}

Possibly the simplest example of an imported limit in $\Lens$, 
which is not a limit in general, is the imported product. 
In the literature, this has previously be called the 
\emph{constant complement lens} \cite{JRW12}.

\begin{proposition}
\label{proposition:imported-product}
The category $\Lens$ has all imported products along the forgetful 
functor to $\Cat$. 
\end{proposition}
\begin{proof}
Given a pair of categories $A$ and $B$, we need to show that the 
projection functors 
(for example, $\pi_{0} \colon A \times B \rightarrow A$)
have a canonical lens structure. 
Using Proposition~\ref{proposition:triangle-rep}, the lens structure
on the projection functor may be represented by the following diagram
in $\Cat$, 
\begin{equation}
\label{equation:imported-product}
\begin{tikzcd}[column sep=small, row sep=small]
& A \times B_{0}
\arrow[ld, "1 \times i"']
\arrow[rd, "\pi_{0}"]
& \\
A \times B
\arrow[rr, "\pi_{0}"'] 
& & 
A
\end{tikzcd}
\end{equation}
where $i \colon B_{0} \rightarrow B$ is the inclusion of the discrete
category $B_{0}$ of objects into $B$. 
More explicitly, the lifting operation on $\pi_{0}$ is given by 
$\varphi((a, b), u \colon a \rightarrow a') = (u, 1_{b})$. 
The above argument extends to imported products indexed by any set. 
\end{proof}

\begin{remark}
In general, the imported product of a pair of categories is \emph{not} 
the cartesian product in $\Lens$, as the corresponding universal 
property does not hold. 
For example, given the imported product $A \times A$, 
there does not exist (in general) a unique lens 
$A \br A \times A$ such that the 
composite with the projections yields identity lenses, since 
a lifting operation 
$\varphi(a \in A, 
(u \colon a \rightarrow x, v \colon a \rightarrow y) \in A \times A)$
is not well-defined unless $u = v$. 
\end{remark}

Despite the above remark, there are instances where the imported 
product in $\Lens$ does coincide with the cartesian product in $\Lens$.

\begin{proposition}
\label{proposition:product-discrete}
The imported product $A \times B$ in $\Lens$ corresponds with the 
cartesian product in $\Lens$ if $A$ or $B$ is a discrete category. 
\end{proposition}
\begin{proof}
Consider the imported product $A \times B_{0}$ where $B_{0}$ is a 
discrete category. 
Then the projection lens $A \times B_{0} \br A$ defined in 
\eqref{equation:imported-product} is a discrete opfibration. 
Thus given any pair of lenses $(f, \varphi) \colon C \br A$ and 
$(g, \gamma) \colon C \br B_{0}$, the canonical functor 
$\langle f, g \rangle \colon C \rightarrow A \times B_{0}$ has 
a unique lens structure which commutes with the projection 
$A \times B_{0} \br A$ by Lemma~\ref{lemma:division}. 
This unique lens structure also 
commutes with the other projection $A \times B_{0} \br B_{0}$. 
Therefore, $A \times B_{0}$ has the universal property of the 
product in $\Lens$. 
\end{proof}

To show that $\Lens$ is distributive, we first need the following 
corollary of Proposition~\ref{proposition:imported-product}.

\begin{corollary}
\label{corollary:monoidal-structure}
The category $\Lens$ has a semi-cartesian symmetric monoidal structure 
given by imported product, and the forgetful functor 
$\U \colon \Lens \rightarrow \Cat$ is strong monoidal. 
\end{corollary}

\begin{proposition}
\label{proposition:distributive}
The category $\Lens$ is a distributive monoidal category with 
respect to the imported product monoidal structure. 
In other words, imported products distribute over coproducts. 
\end{proposition}
\begin{proof}
We need to show that for all categories $A$, $B$, and $C$, the 
canonical lens,
\begin{equation*}
[1 \times \iota_{B}, 1 \times \iota_{C}] 
\colon (A \times B) + (A \times C) \br A \times (B + C) 
\end{equation*}
is an isomorphism, where $\iota_{B} \colon B \br B + C$ and 
$\iota_{C} \colon C \br B + C$ are the coproduct injections. 
Since $\Cat$ is a distributive cartesian monoidal category, 
and the forgetful functor $\U \colon \Lens \rightarrow \Cat$ 
is a strong monoidal isofibration by 
Corollary~\ref{corollary:monoidal-structure}, the result follows 
immediately. 
\end{proof}

\subsection*{Imported pullbacks and extensivity} 

We now turn our attention to imported pullbacks, one of the primary 
motivations for this paper. 

\begin{proposition}
\label{proposition:imported-pullback}
The category $\Lens$ has all imported pullbacks along the forgetful
functor to $\Cat$. 
\end{proposition}
\begin{proof}
Given a cospan of lenses represented as commutative diagrams, 
\begin{equation}
\label{equation:cospan-lenses}
\begin{tikzcd}[column sep=small, row sep=small]
& X
\arrow[ld, "\varphi"']
\arrow[rd, "\phibar"]
& & Y
\arrow[ld, "\gammabar"']
\arrow[rd, "\gamma"]
& \\
A 
\arrow[rr, "f"']
& & B
& & C
\arrow[ll, "g"]
\end{tikzcd}
\end{equation}
we need to show that the pullback projection functors (for example, 
$\pi_{0} \colon A \times_{B} C \rightarrow A$) have a canonical 
lens structure such that $f \circ \pi_{0} = g \circ \pi_{1}$ in $\Lens$. 
Following Proposition~\ref{proposition:triangle-rep}, the lens structure
on the projection functor may be represented by the following diagram
in $\Cat$, 
\begin{equation}
\label{equation:imported-pullback}
\begin{tikzcd}[column sep=small, row sep=small]
& A \times_{B} Y
\arrow[ld, "1 \times \gammabar"']
\arrow[rd, "\pi_{0}"]
& \\
A \times_{B} C
\arrow[rr, "\pi_{0}"'] 
& & 
A
\end{tikzcd}
\end{equation}
where $A \times_{B} Y$ is the pullback of $f$ along $\gammabar$. 
More explicitly, the lifting operation on $\pi_{0}$ is given by:
\begin{equation*}
	((a, c) \in A \times_{B} C, u \colon a \rightarrow a' \in A) 
	\qquad \longmapsto \qquad
	(u, \gamma(c, u)) 
\end{equation*}
Moreover the projection lenses defined above make the appropriate square 
in $\Lens$ commute.
\end{proof}

\begin{example}[Pullbacks as independent components of a state machine]
Consider two state machines $A$ and $C$ over the same interface $B$, 
as lenses $A \br B$ and $C \br B$. 
The imported pullback lens $A \times_{B} C \br B$ models a 
state-machine where the states are pairs $(a \in A, c \in C)$; 
it can be viewed as a state machine with two independent components
$A$ and $C$, which concurrently update according to inputs from 
interface $B$.
\end{example}

There is a close relationship between imported products and 
imported pullbacks. 

\begin{proposition} 
\label{corollary:pullback-over-terminal}
Imported pullbacks over the terminal category correspond to imported products. 
\end{proposition}

We also have the following result, which generalises 
Corollary~\ref{corollary:monoidal-structure}.

\begin{corollary}
\label{corollary:slice-monoidal}
For each category $B$, the category $\Lens / B$ has a semi-cartesian
monoidal structure given by imported pullback, 
and the forgetful functor $\U / B \colon \Lens / B \rightarrow \Cat / B$ 
is strong monoidal. 
\end{corollary}

As with imported products, it is again natural to ask when the 
imported pullback in $\Lens$ coincides with the categorical pullback 
in $\Lens$, leading to the following result which generalises
Proposition~\ref{proposition:product-discrete}.

\begin{proposition}
\label{proposition:pullback-discrete}
The imported pullback $A \times_{B} C$ of the cospan 
\eqref{equation:cospan-lenses} in $\Lens$ corresponds with the 
categorical pullback in $\Lens$ if $f \colon A \rightarrow B$ or 
$g \colon C \rightarrow B$ is a discrete opfibration.
\end{proposition}
\begin{proof}
Suppose $g \colon C \rightarrow B$ in the cospan 
\eqref{equation:cospan-lenses} is a discrete opfibration. 
Since discrete opfibrations are stable under pullback, 
the pullback projection \eqref{equation:imported-pullback} is 
a discrete opfibration. 
Then using Lemma~\ref{lemma:division}, it is straightforward to show
using an analogous argument to the proof of 
Proposition~\ref{proposition:product-discrete} that $A \times_{B} C$
has the universal property of the pullback in $\Lens$. 
\end{proof}

\begin{remark}
It is natural to wonder if all pullbacks in $\Lens$ are of the kind 
described in Proposition~\ref{proposition:pullback-discrete}. 
There are examples where pullbacks exist along lenses which are not 
discrete opfibrations; however the details are outside the scope of
this paper. 
\end{remark}

We are now able to prove the main theorem of this section.

\begin{theorem}
\label{theorem:extensive}
The category $\Lens$ is extensive. 
\end{theorem}
\begin{proof}
By Proposition~\ref{proposition:coproducts}, the category $\Lens$
has finite coproducts. 
By Proposition~\ref{proposition:pullback-discrete}, the category 
$\Lens$ has pullbacks along discrete opfibrations, hence pullbacks 
along coproduct injections. 
Moreover, given any commutative diagram in $\Lens$ of the form, 
\begin{equation}
\label{equation:extensive}
\begin{tikzcd}[column sep=small, row sep=small]
X
\arrow[d]
\arrow[r]
& Z
\arrow[d]
& Y
\arrow[l]
\arrow[d]
\\
A
\arrow[r, "\iota_{A}"']
& A + B
& B
\arrow[l, "\iota_{B}"]
\end{tikzcd}
\end{equation}
the statement that the two squares are pullbacks if and 
only if the top row is a coproduct diagram follows directly, 
since $\Cat$ is extensive and the functor 
$\U \colon \Lens \rightarrow \Cat$ is an identity-on-objects 
isofibration. 
\end{proof}

\section{Conclusion}
This paper has begun the study of the category $\Lens$ whose morphisms 
are lenses between small categories. 
We have presented results about limits, about some imported limits, and 
about coproducts, along with aspects of their interaction including 
extensivity. 
The work has continued apace with important findings by Di Meglio 
\cite{MattACT21} who studies further colimits in $\Lens$. 

The results have been surprising because the category of lenses, which 
is practically important but seemed rather ad hoc, turns out to have many 
aspects which are simpler than $\Cat$, and some aspects which are 
surprisingly like the category of sets. 
In many respects imported limits interact well with one another, and 
with real limits and colimits.

So far we have only studied one category of lenses, but there are many 
more, including ($2$-)categories whose morphisms are symmetric lenses, 
split opfibrations, and discrete opfibrations. 
Future work aims to explore these categories and their interactions with 
$\Lens$, and to further clarify the role played by identity-on-objects 
isofibrations and limits and colimits imported along them.

\bibliographystyle{eptcs}
\bibliography{ACT2021-references.bib}

\end{document}